\documentclass[11 pt]{amsart}

\usepackage{lineno,hyperref}
\usepackage{amsmath,amssymb}
\usepackage{lineno}
\usepackage{centernot}
\theoremstyle{plain}
\newtheorem{thm}{Theorem}[section]
\newtheorem{lemma}[thm]{Lemma}
\newtheorem{cor}[thm]{Corollary}
\newtheorem{prop}[thm]{Proposition}

\theoremstyle{definition}
\newtheorem{defn}[thm]{Definition}
\newtheorem{example}[thm]{Example}

\newtheorem{rem}[thm]{Remark}

\numberwithin{equation}{section}

\newcommand{\interior}[1]{{\kern0pt#1}^{\mathrm{o}}}

\begin{document}


\title{Nil-reversible rings}

\author[Sanjiv Subba]{Sanjiv Subba $^\dagger$}

\address{$^\dagger$Department of Mathematics\\ NIT Meghalaya\\ Shillong 793003\\ India}
\email{sanjivsubba59@gmail.com}
\author[Tikaram Subedi]{Tikaram Subedi  {$^{\dagger *}$}}
\address{$^\dagger$Department of Mathematics\\ NIT Meghalaya\\ Shillong 793003\\ India}
\email{tikaram.subedi@nitm.ac.in}

\subjclass[2010]{Primary 16U80; Secondary 16S34, 16S36}.


\begin{abstract}
This paper introduces and studies nil-reversible rings wherein we call a ring $R$ nil-reversible if the left and right annihilators of every nilpotent element of $R$ are equal. Reversible rings (and hence reduced rings) form a proper subclass of nil-reversible rings and hence we provide some conditions for a nil-reversible ring to be reduced. It turns out that nil-reversible rings are abelian, 2-primal, weakly semicommutative and nil-Armendariz. Further, we observe that the polynomial ring over a nil-reversible ring $R$ is not necessarily nil-reversible in general, but it is nil-reversible if $R$ is Armendariz additionally.

\noindent \textbf{Keywords:} Reversible ring, nil-reversible ring, abelian ring.
\end{abstract}

\maketitle

\section{INTRODUCTION}
 Throughout this paper, unless otherwise mentioned all rings considered are associative with identity, $R$ represents a ring and all modules are unital. For any $a\in R$, the notations $r(a)$ and $l(a)$ represent right and left annihilators of $a$ respectively. We write $Z(R),~P(R),~ J(R),~N(R)$ and $U(R)$ respectively for the set of all central elements, the prime radical, the Jacobson radical, the set of all nilpotent elements, and the group of units of $R$.
 
 \quad A ring or an ideal is said to be reduced if it contains no non-zero nilpotent elements. \textit{Reversible ring} is defined in \cite{rs} and it generalizes reduced rings. $R$ is said to be reversible (\cite{rs}) if for any $w,h \in R$, $wh=0$ implies  $hw=0$. $R$ is said to be \textit{central reversible}(\cite{crev}) if $wh=0$ implies $hw\in Z(R)$, $w,h\in R$. $R$ is \textit{semicommutative} if $wh=0$ implies $wRh=0$, where $w,h\in R$. It is obvious that reversible rings are semicommutative rings. According to \cite{weaklysemi}, a ring $R$ is called \textit{weakly semicommutative} if for any $w,h\in R, wh=0$ implies $wrh\in N(R)$ for any $r\in R$. Clearly, weakly semicommutative rings are generalizations of semicommutative rings.
Following \cite{ucen}, $R$ is called unit-central whenever $U(R)\subseteq Z(R)$.
 
  \quad Following \cite{crev}, $R$ is called \textit{right} (\textit{left}) \textit{principally projective} if the right (left) annihilator of an element of $R$ is generated by an idempotent. A ring in which all idempotent elements are central  is said to be \textit{abelian}.  $R$ is \textit{directly finite} if $xy=1$ implies $yx=1$, $x,y\in R$. It is well known that abelian rings are directly finite. Following \cite{gpv}, $R$ is said to be \textit{strongly regular} if for any $x$ in $R$, there is $y$ in $R$ satisfying $x=yx^2$. $R$ is said to be $2-primal$ if $N(R)$ coincides with $P(R)$. Rings in which $N(R)$ is an ideal are said to be $NI$. Note that 2-primal rings are $NI$. $R$ is called a \textit{left} (\textit{right}) $SF$ (\cite{sfr}) if all simple left (right) $R$-modules are flat. A left $R$-module $M$ is said to be $YJ-$injective (\cite{yj}) if for every $w~(\neq 0)$ in $R$, there is a positive integer $m$ such that $w^m\neq 0$ and every left $R$-homomorphism from $Rw^m$ to $M$ extends to one from $R$ to $M$. Following \cite{gpv}, $R$ is called left $GPV$-ring if every simple left $R$-module is $YJ$-injective. Following \cite{nili}, a left $R$-module $M$ is called $Wnil-injective$ if for any $w~(\neq 0)\in N(R)$, there exists a positive integer $m$ such that $w^m\neq 0 $ and any left $R$-homomorphism $f:Rw^m\rightarrow M$ extends to one from $R$ to $M$. It is obvious that $YJ$-injective modules are left $W$-nil injective.
 
 \quad According to Rege and Chhawchharia (\cite{mb}), a ring $R$ is called \textit{Armendariz}  if  $f(x)=\sum\limits_{i=0}^{n}w_ix^i, g(x)=\sum\limits_{j=0}^{m}h_jx^j \in R[x]$ satisfy $f(x)g(x)=0$, then $w_ih_j=0$ for each $i,j$. In \cite{nilp}, $R$ is called \textit{nil-Armendariz} if $f(x)=\sum\limits_{i=0}^{n}w_ix^i, g(x)=\sum\limits_{j=0}^{m}h_jx^j\in R[x]$ satisfy $f(x)g(x)\in N(R)[x]$, then $w_ih_j\in N(R)$ for each $i,j$. According to \cite{carm}, $R$ is called central Armendariz if for any $f(x)=\sum\limits_{i=0}^{n}w_ix^i,g(x)=\sum\limits_{j=0}^{m}h_jx^j\in R[x]$ satisfy $f(x)g(x)=0$, then $w_ih_j\in Z(R)$ for each $i,j$.
 
 \quad For an endomorphism $\alpha$ of $R$, $R[x;\alpha]$ denotes the  \textit{skew polynomial} ring over $R$ whose elements are polynomials $\sum\limits_{i=0}^{n}w_ix^i,w_i\in R$, where addition is defined as  usual and the multiplication is defined by the rule $xw=\alpha (w)x$ for any $w\in R$. $R[x,x^{-1};\alpha]$ denotes the \textit{skew Laurent polynomial} ring and its elements are finite sum of the form $x^{-i}wx^j$ where $w\in R$ and $i$ and $j$ are non-negative integers. The \textit{skew power series} ring is represented by $R[[x;\alpha]]$ and its elements are of the form  $\sum\limits_{i=0}^{\infty}w_ix^i$, $w_i\in R$ and non-negative integer $i$. The \textit{skew Laurent power series} ring is denoted by $R[[x,x^{-1};\alpha]]$ and it is the localization of $R[[x;\alpha]]$ with respect to $\{x^i\}_{i\geq 0}$. Thus, $R[[x;\alpha]]$ is a subring of $R[[x,x^{-1};\alpha]]$. If $\alpha$ becomes an automorphism of $R$, then the elements of $R[[x,x^{-1};\alpha]]$ are of the form $x^pw_p+x^{p+1}w_{p+1}+...+w_0+w_1x+...$, $w_j\in R$ and integer $p\leq 0$ and $j\geq p$, and addition is defined as usual and the multiplication is defined by the relation $xw=\alpha(w) x$ for any $w\in R$. According to \cite{psr}, $R$ is called \textit{power-serieswise Armendariz} if whenever $f(x)=\sum\limits_{i=0}^{\infty}w_ix^i,g(x)=\sum\limits_{j=0}^{\infty}h_jx^j\in R[[x]]$ satisfy $f(x)g(x)=0$, then $w_ih_j=0$ for all $i,j$. $R$ is said to be \textit{skew power-serieswise Armendariz} (or SPA)(\cite{skew}) if for any $g(x)=\sum\limits_{i=0}^{\infty}w_ix^i, h(x)=\sum\limits_{j=0}^{\infty}h_jx^j\in R[[x;\alpha]], g(x)h(x)=0$ if and only if $w_ih_j=0$ for all $i,j$.
 
 \quad In this paper, we introduce a class of rings called \textit{nil-reversible} rings which is a strict generalization of reversible rings. It is proved that each of the classes of abelian, $2$-primal, weakly semicommutative and nil-Armendariz rings strictly contains the class of nil-reversible rings. It is also shown that the class of nil-reversible rings strictly contains the class of unit central rings. Moreover, the nil-reversibility property is shown to be preserved under various ring extensions.

	\section{Nil-reversible rings}

\begin{defn}\rm
We call a ring $R$ \emph{nil reversible} if $l(a)=r(a)$ for any $a \in N(R)$.
\end{defn}
\begin{rem}\label{sub}
By definition, it is clear that nil-reversible rings are closed under subrings.
\end{rem}

 It is evident that a reversible ring is nil-reversible. However, there exists a nil-reversible ring which is not reversible as given in the next example.  

 \begin{example}
 Consider the ring $R=\mathbb{Z}_{2} [x,y]/I$, where $xy \ne yx$ and $I=< y^2x, xy, yx^2>$. Clearly, $R$ is not reversible. We have $N(R)=\{\overline{0},\overline{yx}\}$ and it is evident that $l(a)=r(a)$ for any $a\in N(R)$. So $R$ is nil-reversible. 
 \end{example}
\begin{prop}\label{a2w}
  Nil-reversible rings are
  \begin{enumerate}
  \item abelian.
  \item 2-primal.
  \item weakly semicommutative.
  \end{enumerate}

\end{prop}
\begin{proof}
 Let $R$ be a nil-reversible ring.
  \begin{enumerate}
  \item Let $e$ be an idempotent element of $R$. For any $x\in R$, $ex-exe\in N(R)$. Note that $(ex-exe)e=0$. By hypothesis, $e(ex-exe)=0$, so $ex=exe $. Again, $xe-exe\in N(R)$ and $e(xe-exe)=0$. So by nil-revesibility of $R$, we have $(xe-exe)e=0$, that is, $xe=exe$. Hence, $ex=xe$.\\
  
  \item Note that $P(R)\subseteq N(R)$. Suppose $w~(\neq 0)\in N(R)$. Then there is a positive integer $m\geq 2$ such that $w^m=0$. Thus, $Rw^{m-1}w=0$. This implies that $wRw^{m-1}=0$ as $R$ is nil-reversible. This yields $(Rw)^m=0$, so $w\in P(R)$.
  
  \item Let $a,b\in R$ be such that $ab=0$. Then, $ba\in N(R)$. For any $r\in R, rbaba=0$. Since $R$ is nil-reversible, $barba=0$. Then, $(arb)^3=arbarbarb=0$, whence $arb\in N(R)$. Hence, $R$ is weakly semicommutative.
  \end{enumerate}
\end{proof}
 One may think whether an abelian or a 2-primal or a weakly semicommutative ring is nil-reversible. Hence, we provide the following examples.
 \begin{example}
  \begin{enumerate}
    \item  Let $\mathbb{Z}$ be the ring of integers, and consider the ring \\$R$=
    $\left\{\left[\begin{array}{rr}
    x & y \\
    z & w\\ 
    \end{array}\right]: x\equiv w(mod~ 2), y\equiv z(mod ~2),x,y,z,w\in \mathbb{Z} \right \}$. By [\cite{csemi},Example 2.7], $R$ is abelian. Now, $\left[\begin{array}{rr}
    0 & 2 \\
    0 & 0\\
    \end{array} \right] \in N(R)$. We have, $\left[\begin{array}{rr}
     0 & 2 \\
     0 & 0\\
     \end{array} \right]\left[\begin{array}{rr}
      2 & 0 \\
      0 & 0\\
      \end{array} \right]=\left[\begin{array}{rr}
       0 & 0\\
       0 & 0\\
       \end{array} \right]$. But  $\left[\begin{array}{rr}
         2 & 0 \\
         0 & 0\\
         \end{array} \right]\left[\begin{array}{rr}
          0 & 2 \\
          0 & 0\\
          \end{array} \right]\neq \left[\begin{array}{rr}
           0 & 0\\
           0 & 0\\
           \end{array} \right]$. So, $R$ is not nil-reversible.
    \item  Take $R=\left[\begin{array}{rr}
    \mathbb{R} & \mathbb{R} \\
     0 & \mathbb{R}\\
     \end{array}\right]$, where $\mathbb{R}$ denotes the filed of real numbers. Then, $R$ is $2-primal$. Observe that $\left[\begin{array}{rr}
           0 & 2 \\
           0 & 0\\
           \end{array}\right] \in N(R) $,
      $\left[\begin{array}{rr}
      2 & 0 \\
      0 & 0\\
      \end{array}\right]\left[\begin{array}{rr}
      0 & 2 \\
      0 & 0\\
      \end{array}\right] \ne 0$ , $\left[\begin{array}{rr}
         0 & 2 \\
         0 & 0\\
         \end{array}\right]\left[\begin{array}{rr}
         2 & 0 \\
         0 & 0\\
         \end{array}\right] = 0$. So, $R$ is not nil-reversible.
   \item We use the ring given in [\cite{onws},example 2.1]. Let $F$ be a field. $F<x,y>$ the free algebra in non-commuting indeterminates $x,y$ over $F$ and  $J$ denotes the ideal $<x^2>^2$ of $F<x,y>$, where $<x^2>$ is the ideal of $F<x,y>$ generated by $x^2$. Let $R=F<x,y>/J$. Then by \cite{onws}, $R$ is weakly semicommutative and $N(R)=\overline{x}R\overline{x}+R\overline{x}^2R+F\overline{x}$. Hence $\overline{x}\in N(R)$. We have, $ \overline{x}(\overline{x}^3\overline{y})=\overline{x^4}\overline{y}=0$. But $(\overline{x}^3\overline{y})\overline{x}=\overline{x^3yx}\neq 0$. Thus, $l(\overline{x})\neq r(\overline{x})$. Therefore, $R$ is not nil-reversible.
   \end{enumerate}
 \end{example}
  Clearly, every reduced ring is nil-reversible, but not every nil-reversible ring is reduced. The next proposition provides some conditions under which a nil-reversible ring becomes reduced.
   \begin{prop}\label{red}
   Let $R$ be a nil-reversible ring. Then $R$ is reduced if $R$ satisfies any one of the following conditions:
   \begin{enumerate}
   \item $R$ is \emph{right (left) principally projective}.
   \item $R$ is \emph{semiprime}.
   \item  R is left SF
   \item  R is right SF
   \item  Every simpler singular left R-module is Wnil-injective
   \item  Every simpler singular right R-module is Wnil-injective
   \item  Every maximal essential left ideal of R is W-nil injective
   \item  Every maximal essential right ideal of R is Wnil-injective 
   
   \end{enumerate}
   \end{prop}
   \begin{proof}
    \begin{enumerate}
    \item  Let $R$ be a right principally projective ring. Suppose $w \in R$ such that $w^2=0$. Then, $w \in N(R)$. As $R$ is right principally projective, $r(w)=eR $ for some idempotent $e \in R$. This implies $w=ew $. As $e\in r(w)$, $we=0$. Since $R$ is nil-reversible, $ew
    =0$. Thus, $w=0$ and so $R$ is reduced.
    Similarly, we can prove that $R$ is reduced if $R$ is left principally projective.  
    \item   Let $R$ be a semiprime ring. Suppose $w \in R$ such that $w^2=0$. Then $Rww=0$. Since $R$ is nil-reversible, $wRw=0$. By hypothesis, $w=0$. Thus, $R$ is reduced.
    \item Let $w~(\neq 0)\in R$ and $w^2=0$. Then $l(w)\neq R$ and $l(˘w)$ is contained in some maximal left ideal $H$ of $R$. As $w\in l(w)\subseteq H$ and $R$ is left $SF$, there exists $h\in H$ such that $w=wh$. So, $w(1-h)=0$. Since $R$ is nil-reversible, $(1-h)w=0$, $(1-h)\in l(w)\subseteq H$. Then, $1=(1-h)+h\in H$. This is a contradiction.
    \item The proof is similar to (3).
    \item Let $w~(0\neq)$ be in $R$ satisfying $w^2=0$. Then, $l(w)$ is contained in some maximal left ideal $H$. If possible, let us assume that $H$ is not essential. Then there exists an idempotent $e\in R$ with the property $H=l(e)$. Since $w\in H$, $we=0$. As $R$ is abelian (by \textbf{Proposition \ref{a2w}}), $ew=0$, so $e\in l(w)\subseteq H=l(e) $. Thus, $e^2=0=e$. This is a contradiction. Therefore, $H$ is essential and $R/H$ is simple singular left $R$-module which is $W$nil-injective by hypothesis. Define $f: Rw\rightarrow R/H$ by $f(rw)=r+H$. This is well-defined as $l(W)\subseteq H$. By hypothesis, there exists a left $R$-homomorphism $g:R\rightarrow H$ which extends $f$.  Therefore, $1+H=f(w)=g(w)=ws+H$, for some $s\in R$. So, $1-ws\in H$. Now, $w^2s=0$ implies $ws\in r(w)=l(w)\subseteq H$. So, $ws\in H$, $1=1-ws+ws\in H$. This is a contradiction.
    \item The proof is same as in (5).
    \item Let $w~(\neq 0) \in R$ and $w^2=0$. Then, $l(w)\subseteq H$ for some maximal left ideal $H$ of $R$. As in (5), $H$ is essential. Define $f:Rw\rightarrow H$ as $f(rw)=rw$. By hypothesis, there exists a left $R$-homomorphism $g:R\rightarrow H$ which extends $f$. Therefore, $w=f(w)=g(w)=wg(1)=wh$ for some $h\in H$. As $R$ is nil-reversible, $1-h\in l(w)\subseteq H$. So, $1=(1-h)+h\in H$, which is a contradiction.
    \item Similar to (7).
    \end{enumerate}
    \end{proof}
    \begin{cor}
    Nil-reversible left (right) $SF$ ring is strongly regular.
    \end{cor}
    \begin{proof}
    Follows from \cite[Theorem 4.10]{rege} and the \textbf{Proposition \ref{red}}.
    \end{proof}
      Let $P$ be an (R,R)-bimodule. The \textit{trivial extention} of $R$ by $P$ is the ring $T(R,P)=R\bigoplus P$, where the addition is usual and the multiplication is given by:\\
             $(r_1,p_1)(r_2,p_2)=(r_1r_2,r_1p_2+p_1r_2), p_i\in P,r_i\in R $ and $i=1,2$. This ring is isomophic to the ring \\ $\left\{\left[\begin{array}{rr}
             x & p\\
             0 & x
             \end{array}
             \right] : x\in R, p\in P \right\}$, where the operations are usual matrix operations.
                \begin{prop}
                    If $T(R,R)$ is nil-reversible then $R$ is reversible.
                \end{prop}
                 \begin{proof}
                 Let $w,h\in R$ and $wh=0$. Observe that $\left[\begin{array}{rr}
                  0 & w \\
                  0 & 0\\
                  \end{array}\right] \in N(T(R,R))$ and $\left[\begin{array}{rr}
                   h & 0 \\
                   0 & h\\
                   \end{array}\right] \in T(R,R)$ and
                   $\left[\begin{array}{rr}
                    0 & w \\
                    0 & 0\\
                    \end{array}\right]\left[\begin{array}{rr}
                     h & 0 \\
                     0 & h\\
                     \end{array}\right]=\left[\begin{array}{rr}
                      0 & wh \\
                      0 & 0\\
                      \end{array}\right]=\left[\begin{array}{rr}
                       0 & 0 \\
                       0 & 0\\
                       \end{array}\right]$. Since $T(R,R)$ is nil-reversible,
                       $\left[\begin{array}{rr}
                         0 & 0 \\
                         0 & 0\\
                         \end{array}\right]=\left[\begin{array}{rr}
                          h & 0 \\
                          0 & h\\
                          \end{array}\right]\left[\begin{array}{rr}
                           0 & w \\
                           0 & 0\\
                           \end{array}\right]=\left[\begin{array}{rr}
                            0 & hw \\
                            0 & 0\\
                            \end{array}\right]$. This gives $hw=0$.
                 \end{proof}
                 
\begin{prop}
 \label{sq}
 If $R$ be a ring such that $(wh)^2=0$ implies $wh=0$ and $(hw)^2=0$ implies $hw=0,w\in N(R),h\in R$ then $R$ is nil-reversible.
 \end{prop}
 \begin{proof}
 Let $wh=0,w\in N(R), h\in R$. Then $(hw)^2=0$. So by hypothesis, $hw=0$. This implies $r(w)\subseteq l(w)$. Similarly, $l(w)\subseteq r(w)$ for $w\in N(R)$. Thus, $R$ is nil-reversible.
 \end{proof}
 \begin{prop}
If $R$ is unit-central then $R$ is nil-reversible.
 \end{prop}
 \begin{proof}
 Let $w\in N(R), h\in R$. Now, $(1-w)\in U(R)\subseteq Z(R)$. So, $(1-w)(1-h)=(1-h)(1-w)$. This implies $hw=wh$. Therefore, $hw=0$ if and only if $wh=0$. Hence $R$ is nil-reversible.
 \end{proof}
 But the converse does not hold.
 \begin{example}
Take $R=T(\mathbb{R},\mathbb{R})$ and $a=\left[\begin{array}{rr}
1 & 1\\
0 & 1
\end{array} \right]\in R$, where $\mathbb{R}$ denotes the field of real numbers. Then $R$ is nil-reversible by \cite[Proposition 1.6]{ext}. Obviously, $a\in U(R)$ but $a\notin Z(R)$. 
 \end{example}
 \begin{prop}
           Suppose $R$ is a central reversible ring containing $0$ as the only central nilpotent element then $R$ is nil-reversible.
         \end{prop}
         \begin{proof}
            Let $a\in N(R)$ and $b\in r(a)$. We have $a^m=0$ for some positive integer $m$. Since $R$ is central reversible, $ba\in Z(R)$.
            Now, $(ba)^m=b^ma^m=0$. Thus, $ba\in N(R)\cap Z(R)=\{0\}$. It gives $ba=0$, and so $r(a)\subseteq l(a)$. Similarly, $l(a)\subseteq r(a)$. Hence $r(a)=l(a)$.
            \end{proof}
  Let $A$ be an algebra (not necessarily with identity) over a commutative ring $C$. The \emph{Dorroh extension} of $A$ by $C$ is the ring denoted by $A \bigoplus _{D} C$, with the operations $(a_1,s_1)+(a_2,s_2)=(a_1+a_2,s_1+s_2)$ and $(a_1,s_1)(a_2,s_2)=(a_1 a_2+s_{1}a_{2}+s_{2}a_{1},s_{1}s_{2})$ for all $a_i\in A$ and $s_i\in C, i=1,2$.
  \begin{prop}
  \textit{Dorroh Extension} of a nil-reversible ring $R$ by $\mathbb{Z}$ is nil-reversible.
  \end{prop}
  \begin{proof}
    Let $(a,0) \in N(R\bigoplus_D \mathbb{Z})$ and $(s,m) \in l(a,0)$.  Then, $(0,0)=((s+m)a,0)$, hence  $(s+m)a=0$. Since $R$ is nil-reversible, $(s+m) \in r(a)$. So $(a,0)(s,m)=(0,0)$, $l(a,0)\subseteq r(a,0)$. Similarly, $r(a,0) \subseteq l(a,0).$ Hence, $R\bigoplus_D \mathbb{Z}$ is nil reversible.
  \end{proof}
  \begin{prop}\label{subdi}
      The subdirect product of arbitrary family of nil-reversible rings is nil-reversible.
  \end{prop}
  \begin{proof}
   Let $\{I_{\lambda}| \lambda\in \Delta \}$ be a family of ideals of $R$ such that $\cap_{\lambda \in \Delta}I_\lambda=0$ and $R/I_\lambda$'s are nil-reversible rings, where $\Delta$ is some index set. Take $a\in N(R)$ and $b\in R$ satisfying $ab=0$. There is a positive integer $m$ with the property $a^m=0$. Thus, $a+I_\lambda \in N(R/I_\lambda)$. Since $N(R/I_\lambda)$ is nil-reversible and $(a+I_\lambda)(b+I_\lambda)=0$,~$(b+I_\lambda)(a+I_\lambda)=0$. This implies $ba+I_\lambda=0$. So $ba\in I_\lambda$ for all $\lambda\in \Delta$, i.e., $ba\in \cap_{\lambda \in \Delta}I_\lambda=0$. Following similar argument, $ba=0$ implies $ab=0$. Therefore, the subdirect product of nil-reversible rings is nil-reversible.
  \end{proof}
   \begin{prop}\label{frac}
            Let $J$ be a reduced ideal of a ring $R$ such that $R/J$ is nil-reversible. Then $R$ is nil-reversible.
    \end{prop}
      \begin{proof}
            Let $w\in N(R)$ and $h\in R$ be such that $wh=0$. This implies $\bar{w}\bar{h}=0$. As $\bar{w}\in N(R/J)$ and $R/J$ is nil-reversible, we have $\bar{h}\bar{w}=0$. Thus, $hw\in J\cap N(R)$. Since $J$ is reduced, we have $hw=0$. As above, $hw=0$ yields $wh=0$. Thus, $R$ is nil-reversible. 
            \end{proof}
        In \textbf{Proposition \ref{frac}}, reducedness of $J $ is necessary. As shown by the following example that in deficiency of this condition the \textbf{Proposition \ref{frac}} is unlikely to be guaranteed.
        \begin{example}
            Consider the ring\\ $R=\left[\begin{array}{rr}
         F & F \\
         0 & F\\
         \end{array}\right]$, where $F$ is a field.
          Let $e_{ij}$ be matrix units with $1$ at the entry $(i,j)$  and zero elsewhere and $J=e_{ij}R$. Then $J=\left[\begin{array}{rr}
         0 & F \\
         0 & 0\\
         \end{array}\right]$. One can easily verify that $J$ is not a reduced ideal . Now, consider the quotient ring 
         $R/J=\left\{\left[\begin{array}{rr}
         x & 0 \\
         0 & y\\
         \end{array}\right]+J:x,y \in F\right\}$.  Clearly, $R/J$ is reduced and so nil-reversible.
         Observe that $e_{12}^2=0$, $e_{12} \in N(R)$.\\ $e_{12}e_{11}=0$. This implies $e_{11} \in r(e_{12})$. But
         $e_{11}e_{12}=e_{12} \ne 0$. Thus, $r(e_{12})\neq l(e_{12})$. 
        \end{example}
           A homomorphic image of a nil-reversible ring may not be nil-reversible.
          \begin{example}
          Let $R=D[x]$, where $D$ is a division ring and $J=<xy>$, where $xy\neq yx$. As $R$ is a domain, $R$ is nil-reversible. Clearly $\overline{yx}\in N(R/J)$ and $\overline{x}(\overline{yx})=\overline{xyx}=0$. But, $(\overline{yx})\overline{x}=\overline{{yx}^2}\neq 0$. This implies $R/J$ is not nil-reversible. 
          
          \end{example}
          \begin{prop}
            If $R$ is a nil-reversible ring and $J$ an ideal consisting of nilpotent elements of bounded index ~$\leq m$ in $R$, then $R/J $ is nil-reversible.
          \end{prop}
          \begin{proof}
            Let $\overline{a}\in N(R/J),\overline{b} \in R/J$ such that $\overline{a}\overline{b}=0$. Then $ab\in J$ and $\overline{a}\in N(R/J)$. This implies $a\in N(R)$. So $(ab)^m=0=a(ba)^{m-1}b$ for some positive integer $m$. Since $R$ is nil-reversible, $(ba)^m=0$ implying $ba\in J$ and $r(\overline{a})\subseteq l(\overline{a})$. Proceed similarly to get $l(\overline{a})\subseteq r(\overline{b})$. Therefore $R/J$ is nil-reversible.
          \end{proof}

	\section{Polynomial and power series extensions of nil-reversible rings}
	The polynomial ring over a reduced ring is reduced. However, the polyynomial ring over a reversible ring may not be semicommutative by  \cite[ Example 2.1]{ext}.
	   Based on this observation, it is quite natural to question whether the polynomial ring over a nil-reversible ring is nil-reversible. But the following example gives the answer in negative.
 \begin{example} 
     Take $\mathbb{Z}_2$ as the field of integers modulo 2 and let $\Omega=\mathbb{Z}_2[w_0,w_1,w_2, h_0,h_1,h_2,d]$ be the free algebra of polynomials with zero constant terms in non-commuting indeterminates $w_0,w_1,w_2,h_0,h_1,h_2$ and $d$ over $\mathbb{Z}_2.$ Take an ideal $J$ of the ring $\mathbb{Z}_2+\Omega$ generated by $w_0h_0, w_0h_1+w_1h_0, w_0h_2+w_1h_1+w_2h_0, w_1h_2+w_2h_1, w_2h_2,w_0ph_0, w_2ph_2, h_0w_0, h_0w_1+h_1w_0, h_0w_2+h_1w_1+h_2w_0, h_1w_2+h_2w_1,h_2w_2,h_0pw_0,h_2pw_2, (w_0+w_1+w_2)p(h_0+h_1+h_2),(h_0+h_1+h_2)p(w_0+w_1+w_2) ~and~ p_1p_2p_3p_4$ where $p,p_1,p_2,p_3,p_4\in \Omega$. Take $R=\frac{\mathbb{Z}_2+\Omega}{J}$. Then we have $R[x]\cong \frac{(\mathbb{Z}_2+\Omega)[x]}{J[x]}$. Now, by \cite[Example 2.1]{ext}, $R$ is reversible and hence nil-reversible. Now, 
       $(\overline{w_0+w_1x+w_2x^2})\in N(R[x])$ and $(\overline{h_0+h_1x+h_2x^2})\in R[x]$ and $d(h_0+h_1x+h_2x^2)(w_0+w_1x+w_2x^2)\in J[x]$ which means $\overline{d}(\overline{h_0+h_1x+h_2x^2})\in l(\overline{w_0+w_1x+w_2x^2})$. But $(w_0+w_1x+w_2x^2)d(h_0+h_1x+h_2x^2)\notin J[x]$, since $w_0dh_1+w_1dh_0\notin J$. This implies $\overline{d}(\overline{h_0+h_1x+h_2x^2})\notin r(\overline{w_0+w_1x+w_2x^2})$. Thus, $R[x]$ is not nil-reversible.
   \end{example}
     Given a nil-reversible ring $R$, we investigate some conditions under which $R[x]$ also becomes nil-reversible. 
      \begin{prop}
                  If $R$ is an Armendariz ring, then $R$ is nil-reversible if and only if $R[x]$ is nil-reversible.
          \end{prop}
             \begin{proof}
                  Suppose $R$ is nil-reversible.
                  Let $g(x)=\sum\limits_{i=0}^{m}w_ix^i\in N(R[x])$, $h(x)=\sum\limits_{j=0}^{n}h_jx^j\in R[x]$ be such that $g(x)h(x)=0$. By \cite[Corollary 5.2]{nilp}, $N(R[x])=N(R)[x]$. So $w_i\in N(R)$ for all $i$. Since $R$ is Armendariz, $w_ih_j=0$ for all $i,j$. By nil-reversibility, $h_jw_i=0$ for all $i,j$. So, $h(x)g(x)=0$. Following the similar argument, $h(x)g(x)=0$ implies $g(x)h(x)=0 $. Therefore, $R[x]$ is nil-reversible. The proof of the  converse is trivial.    
                  \end{proof}
      \begin{prop}
                 Let $\Delta$ denotes a multiplicatively closed subset of $R$ consisting of central non-zero divisors. Then $R$ is  nil-reversible if and only if $\Delta^{-1}R$ is nil-reversible. 
                 \end{prop}
                  \begin{proof}
                    Suppose $R$ is nil-reversible. Let $s^{-1}a\in N(\Delta^{-1} R).$ Then $(s^{-1}a)^n=0$ for some positive integer $n$. This implies $a^n=0$, so $a\in N(R)$. For any $(s_1^{-1}b)\in \Delta^{-1}R$ having the property that $(s_1^{-1}b)(s^{-1}a)=0$, we have $(ss_1)^{-1}ba=0$, $ba=0$. Since $R$ is nil-reversible, $ab=0$, so $s^{-1}s_1^{-1}ab=0$ which further yields $(s^{-1}a)(s_1^{-1}b)=0$. So, $l(s^{-1}a)\subseteq r(s^{-1}a)$. Similarly, $r(s^{-1}a)\subseteq l(s^{-1}a)$. Hence $\Delta^{-1}R$ is nil-reversible.\\
                    The converse part is trivial.
                    
                    \end{proof}
                    \begin{cor}
                    \label{lau}
                         $R[x]$ is nil-reversible if and only if $R[x,x^{-1}]$ is nil-reversible.
                    \end{cor}
                     \begin{proof}
                       Suppose $R[x]$ is nil-reversible. Take $\Delta=\{1,x,x^2,x^3,...\}\subseteq R[x]$. Clearly, $\Delta$ is multiplicatively closed subset consisting of central  non-zero divisors of $R[x]$ and $R[x,x^{-1}]=\Delta^{-1}R[x]$. So $R[x,x^{-1}]$ is nil-reversible. \\The proof of the converse is trivial.
                       \end{proof}
      \begin{cor}
      For an Armendariz ring $R$, the following are equivalent;
         \begin{enumerate}
         \item $R$ is nil-reversible.
         \item $R[x]$ is nil-reversible.
         \item $R[x,x^{-1}]$ is nil-reversible.
         \end{enumerate}
      \end{cor}
    \begin{prop}
      For a right principally projective ring $R$ and a natural number $m\geq 2$, $R$ is nil-reversible if and only if $\frac{R[x]}{<x^m>}$ is central Armendariz.
    \end{prop}
      \begin{proof}
        Suppose $R$ is nil-reversible. $R$ is reduced by \textbf{Proposition \ref{red}}. From  \cite[Theorem 5]{arm}, $\frac{R[x]}{<x^m>}$ is central Armendariz.
        Conversely, consider $\frac{R[x]}{<x^m>}$ as central Armendariz. By the hypothesis and  \cite[Theorem 2.1]{carm}, $\frac{R[x]}{<x^m>}$ is Armendariz. \cite[Theorem 5]{arm} implies that $R$ is reduced and so nil-reversible.
        \end{proof}
     \begin{prop}
     Let $R$ be a nil-reversible ring. If $f_1,f_2,...,f_n\in R[x]$ be such that $C_{f_1f_2...f_n}\subseteq N(R)$, then $C_{f_1}C_{f_2}...C_{f_n}\subseteq N(R)$, where $C_f$ represents the set of coefficients of $f$.
     \end{prop}  
     \begin{proof}
     From the \textbf{Proposition \ref{a2w}}, $N(R)$ is an ideal of $R$. Clearly, $R/N(R)$ is reduced, and so $R/N(R)$ is Armendariz. Let $f(x)=\sum\limits_{i=0}^{m}w_ix^i\in R[x]$. Define $\zeta: R[x]\rightarrow \frac{R}{N(R)}[x]$ via $\zeta (f(x))=\bar{f}(x)$, where $\bar{f}(x)=\sum\limits_{i=0}^{m}(w_i+N(R))x^i $. Clearly, $\zeta $ is a ring homomorphism. As $C_{f_1f_2...f_n}\subseteq N(R)$, $\bar{f_1}\bar{f_2}...\bar{f_n}=\overline{f_1f_2...f_n}=0$. So, by \cite[Proposition 1]{arm}, $\bar{p_1}\bar{p_2}...\bar{p_n}=0,~ p_i\in C_{f_i},~i=1,2,...,n$. Now, it is easy to verify that $C_{f_1}C_{f_2}...C_{f_n}\subseteq N(R)$. 
     \end{proof}
     \begin{cor}\label{coef}
     If $R$ is nil-reversible ring, then $N(R[x])\subseteq N(R)[x]$.
     \end{cor}
     \begin{prop}
     \label{eql}
     If $R$ is nil-reversible ring, then $N(R[x])=N(R)[x]$.
     \end{prop}
     \begin{proof}
     By \textbf{Corollary \ref{coef}}, $N(R[x])\subseteq N(R)[x]$. Let $f(x)=\sum\limits_{i=0}^{n}w_ix^i\in N(R)[x]$. For each $i$, there exists a positive integer $s_i$ with $w_i^{s_i}=0$. Take $l=s_0+s_1+...+s_n+1$. Now,
     $(f(x))^l=(w_0+w_1x+...+w_nx^n)^l=\sum\limits_{p=0}^{nl}\left(\sum\limits_{i_1+...+i_l=p}w_{i_1}w_{i_2}...w_{i_l}\right)x^p$, where $w_{i_1}, w_{i_2},...,w_{i_l}\in \{w_0,...,w_n\}$. If the number of $w_0$'s in the product $w_{i_1}...w{i_l}$ is more than $s_0$, then we can express $w_{i_1}...w{i_l}$ as $h_0w_0^{j_1}h_1w_0^{j_2}h_2...h_{t-1}w_0^{j_t}h_t$, where $t$ is a non-negative integer and $1\leq j_1,j_2,...,j_t, s_0\leq j_1+j_2+...+j_t$ and for each $i,~ 0\leq i \leq t$, $h_i$ is a product of some elements choosen from $\{w_0,...,w_n\}$ or equal to 1. Here, we have $w_0^{j_1+...+j_l}=0,~w_0^{j_1}(w_0^{j_2}...w_0^{j_l}h_0)=0 $. Since $R$ is nil-reversible, $w_0^{j_2}...w_0^{j_l}h_0w_0^{j_1}=0=w_0^{j_2}(w_0^{j_3}...w_0^{j_l}h_0w_0^{j_1})h_1$. This implies $(w_0^{j_3}...w_0^{j_l}h_0w_0^{j_1})h_1w_0^{j_2}=0$. Continuing this way, we get $h_0w_0^{j_1}...w_0^{j_t}h_t=0$. Thus, $w_{i_1}...w_{i_l}=0$.
     If the number of $w_i$'s in $w_{i_1}...w_{i_l}$ is more than $s_i$'s, then proceeding as above, we get $w_{i_1}...w_{i_l}=0$. Hence $\sum\limits_{i_1+i_2+...+i_l=p}w_{i_1}w_{i_2}...w_{i_l}=0$, $(f(x))^l=0$, $N(R)[x]\subseteq N(R[x])$. Therefore, $N(R)[x]=N(R[x])$.
     \end{proof}
     \begin{prop}
     \label{nilA}
     Nil-reversible rings are nil-Armendariz.
     \end{prop}              	
 \begin{proof}
 Let $R$ be a nil-reversible ring and $f(x)=\sum\limits_{i=0}^{n}w_ix^i,g(x)=\sum\limits_{j=0}^{m}h_jx^j\in R[x]$ and $f(x)g(x)\in N(R)[x]$. Then $f(x)g(x)=\sum\limits_{s=0}^{m+n}\left(\sum\limits_{i+j=s}w_ih_j\right)x^s$, $\sum\limits_{i+j=s}w_ih_j\in N(R)$. Now, we use induction to show that $w_ih_j\in N(R)$ for each $i,j$. For $s=0$,
 $\sum\limits_{i+j=0}w_ih_j=w_0h_0$ $\in N(R)$ which implies $h_0w_0\in N(R)$. For $s=1$,
 $w_0h_1+w_1h_0\in N(R)$. $R$ is $2-primal$ (by \textbf{Proposition \ref{a2w}}), $w_0h_1w_0+w_1h_0w_0\in N(R)$. As $h_0w_0\in N(R)$, $w_0h_1w_0\in N(R)$. Thus, $(w_0h_1)^2\in N(R)$ and so $w_0h_1\in N(R)$, $h_1w_0\in N(R)$. It immediately implies that $w_1h_0\in N(R)$ and so $h_0w_1\in N(R)$. Assume $w_ih_j\in N(R)$ for $i+j=t$, $2\leq t\leq m+n-1$. Now, $\sum\limits_{i+j=t+1}w_ih_j=w_{i_0}h_{j_l}+w_{i_1}h_{j_{l-1}}+...+w_{i_l}h_{j_o}\in N(R)$, where $0\leq l, i_k\in \{0,1,2,...,n\}, j_k\in \{0,1,...,m\}, k\in \{0,1,...,l\}$ and $i_0\leq i_1\leq ...\leq i_l, j_0\leq j_1\leq ....\leq j_l$. If there is only one term in $\sum\limits_{i+j=t+1}w_ih_j$, then we are through. Suppose $\sum\limits_{i+j=t+1}w_ih_j$ contains more than one term. Then we can take either $j_l>i_0$ or $i_0>j_l$. Firstly, take $j_l>i_0$.
 Now, $\sum\limits_{i+j=t+1}w_ih_j=w_{i_0}h_{j_l}+w_{i_1}h_{j_{l-1}}+...+w_{i_l}h_{j_0}\in N(R)$. This implies that $w_{i_0}h_{j_l}w_{i_0}+w_{i_1}h_{j_{l-1}}w_{i_0}+...+w_{i_l}h_{j_0}w_{i_0}\in N(R)$. But $j_{l-1}+i_0\leq t, j_{l-2}+i_0\leq t,...,j_0+i_0\leq t$. Thus, by induction hypothesis $h_{j-1}w_{i_0},...,h_{j_0}w_{i_0}\in N(R)$. Therefore, $w_{i_0}h_{j_l}w_{i_0}\in N(R)$, $w_{i_0}h_{j_l}\in N(R)$. Proceeding similarly, we get that $w_ih_{j}\in N(R)$ for each $i,j$ satisfying $i+j=t+1$. Similar approach can be applied if $i_0>j_l$. Hence $R$ is nil-Armendariz.
 \end{proof}
  \begin{prop}
  If $R$ is nil-reversible, then R[x] is nil-Armendariz.
  \end{prop}
  \begin{proof}
  Let $R$ be nil-reversible. By \textbf{Proposition \ref{nilA}}, $R$ is nil-Armendariz and $N(R[x])=N(R)[x]$ (\textbf{Proposition \ref{eql}} ). Therefore, by  \cite[Theorem 5.3]{nilp}, $R[x]$ is nil-Armendariz.
  \end{proof}
  \begin{lemma}
  \label{semi}
  Let $R$ be a ring such that $(wh)^2=0$ implies $wh=0$ and $(hw)^2=0$ implies $hw=0$, $w\in R,h\in N(R)$. Then, we have the following;
  \begin{enumerate}
  \item If $xyx=0$ or $xy^2=0$ then $xy=0$ for any $x\in R, y\in N(R)$.
  \item If $y^2x=0$ then $yx=0$ for any $x\in R,y\in N(R)$. 
  \end{enumerate}
  \end{lemma}
  \begin{proof}
  By \textbf{Proposition \ref{sq}}, $R$ is nil-reversible. Let $x\in R$ and $y\in N(R)$.
  \begin{enumerate}
   \item Suppose $xyx=0$. This implies $(xy)^2=0$. By hypothesis $xy=0$. If $xy^2=0$, then $yxy=0$ ( by nil-reversibility of $R$ ). So $(xy)^2=0$. This implies $xy=0$. 
   \item By similar computation as above, $y^2x=0$ yields $yx=0$.
   \end{enumerate}
  \end{proof}
  \begin{prop}
  If $R$ is semicommutative in which $(wh)^2=0$ implies $wh=0$ and $(hw)^2=0$ implies $hw=0$ for $w\in R,h\in N(R)$, then both $R[x]$ and $R[[x]]$ are nil-reversible.
  \end{prop}
  \begin{proof}
  By \textbf{Proposition \ref{sq}}, $R$ is nil-reversible. Recall that nil-reversible rings are closed under subrings ( by \textbf{Remark} \ref{sub}), so it suffices to prove that $R[[x]]$ is nil-reversible. Let $g(x)=\sum\limits_{i=0}^{\infty}w_ix^i\in R[[x]], h(x)=\sum\limits_{j=0}^{\infty}h_jx^j\in N(R[[x]])$ and $g(x)h(x)=0$. Since $R$ is nil-reversible, $N(R)$ is an ideal by \textbf{Proposition \ref{a2w}(2)}. By \cite[Proposition 1]{power}, $N(R[[x]])\subseteq N(R)[[x]]$. So $h_j\in N(R)$ for each $j$. Since $g(x)h(x)=0$, we get 
  \begin{equation} \label{1}
  \sum_{k=0}^{\infty}\left(\sum_{i+j=k}w_ih_j\right)x^k=0
  \end{equation} 
  Now, we claim that $w_ih_j=0$ for all $i,j$. We use induction on $i+j$ to prove our claim. Observe that $w_0h_0=0$. This shows that the claim holds true for $i+j=0$. Now, suppose that the claim holds true for $i+j\leq m-1$. From the equation \ref{1}, we get 
  \begin{equation}\label{2}
  \sum_{i=0}^{m}w_ih_{m-i}=0
  \end{equation}
  Multiply the equation \ref{2} by $h_0$ from the right. Then, we obtain $(w_0h_m+w_1h_{m-1}+...+w_mh_0)h_0=0$ which leads to $w_mh_0^2=0$ as $R$ is semicommutative. So $w_mh_0=0$ by the \textbf{Lemma \ref{semi}}. Then the equation \ref{2} becomes
  \begin{equation}\label{3}
  w_0h_m+w_1h_{m-1}+...+w_{m-1}h_1=0
  \end{equation}
  Multiplying the equation \ref{3} by $h_1,h_2,...,h_m$ on the right, we get $w_{m-1}h_1,..., \\ w_1h_{m-1}, w_0h_m=0$. Following similar argument, we get $w_ih_j=0$ for $i+j=m$. Thus, by induction we get $w_ih_j=0$ for all $i,j$. As $R$ is nil-reversible $h_jw_i=0$ for all $i,j$. Hence $h(x)g(x)=0$. Similarly, one can show that $h(x)g(x)=0$ implies $g(x)h(x)=0$. Thus, $R[[x]]$ is nil-reversible.
  \end{proof} 
  \begin{prop}
  Let $R$ be a power serieswise Armendariz ring. Then the following are equivalent:
  \begin{enumerate}
  \item $R$ is nil-reversible.
  \item $R[x]$ is nil-reversible.
  \item $R[[x]]$ is nil-reversible.
  \end{enumerate}
  \end{prop}
  \begin{proof}
  It is sufficient to show that $R[[x]]$ is nil-reversible if $R$ is nil-reversible. By \cite[Lemma 2]{power}, $N(R[[x]])\subseteq N(R)[[x]]$. Let $g(x)=\sum\limits_{i=0}^{\infty}w_ix^i\in N(R[[x]]),h(x)=\sum\limits_{j=0}^{\infty}h_jx^j\in R[[x]]$ and $g(x)h(x)=0$. By hypothesis $w_ih_j=0$. Since $R$ is nil-reversible and $N(R[[x]])\subseteq N(R)[[x]]$, $h_jw_i=0$ for all $i,j$. So, $h(x)g(x)=0$. Proceeding as above, we can show that if $h(x)g(x)=0$ then $g(x)h(x)=0$. Therefore $R[[x]]$ is nil-reversible.
  \end{proof}
  \begin{prop}
  Let $R$ be an SPA ring and $\alpha $ an automorphism of $R$. Then the following are equivalent:
  \begin{enumerate}
  \item $R$ is nil-reversible.
  \item $R[x;\alpha]$ is nil-reversible.
  \item $R[x,x^{-1}; \alpha]$ is nil-reversible.
  \item $R[[x;\alpha]]$ is nil-reversible.
  \item $R[[x,x^{-1};\alpha]]$ is nil-reversible.
  \end{enumerate}
  \end{prop}
  \begin{proof}
  It is enough to show that $(1)\implies (5)$. Suppose that $R$ is nil-reversible. Let $g(x)=x^sw_s+x^{s+1}w_{s+1}+...+w_0+w_1x+...\in N(R[[x,x^{-1}; \alpha]]), h(x)=x^th_t+x^{t+1}h_{t+1}+...+h_0+h_1x+...\in R[[x,x^{-1};\alpha]]$ and $g(x)h(x)=0$, where $s$ and $t$ are integers $\leq 0$. Then $w_ih_j=0$ by  \cite[Proposition 2.8]{skew}. Moreover, $w_i\in N(R)$ for all $i\geq s$ by  \cite[Lemma 3.10 (2)]{nilcom}. Then by nil-reversibility of $R$, $h_jw_i=0$ and  by  \cite[Lemma 3.10 (1)]{nilcom}, $h_j\alpha ^m(w_i)=0$ for any non-negative integer $m$. Thus, $h(x)g(x)=0$. Similarly, $h(x)g(x)=0$ implies $g(x)h(x)=0$. Hence, $R$ is nil-reversible.
  
  \end{proof} 	
	
\end{document}